\documentclass[reqno,11pt]{amsart}
\usepackage{graphicx} 
\usepackage{amsfonts, amscd, epsfig, amsmath, amssymb,enumerate,bm}
\usepackage{color}
\usepackage{mathrsfs}
\usepackage[margin=3cm]{geometry}
\usepackage[colorlinks=true,linkcolor=blue,citecolor=magenta]{hyperref}
\usepackage[displaymath, mathlines,pagewise]{lineno}
\usepackage{setspace}
\numberwithin{equation}{section}

\newtheorem{theorem}{Theorem}[section]
\newtheorem{lemma}[theorem]{Lemma}
\newtheorem{proposition}[theorem]{Proposition}

\newtheorem{remark}[theorem]{Remark}

\newcommand{\E}{\mathbb{E}}
\newcommand{\Pbb}{\mathbb{P}}
\newcommand{\R}{\mathbb{R}}
\newcommand{\Z}{\mathbb{Z}}

\newcommand{\Scal}{\mathcal{S}}
\newcommand{\Ycal}{\mathcal{Y}}

\title[Moderate deviation principles]{Moderate deviation principles for the current and the tagged particle in the WASEP}

\author{Xiaofeng Xue}
\address{School of Mathematics and Statistics, Beijing Jiaotong University, Beijing 100044, China.}
\email{xfxue@bjtu.edu.cn}

\author{Linjie Zhao}
\address{School of Mathematics and Statistics \& Hubei Key Laboratory of Engineering Modeling and Scientific Computing, Huazhong University of Science and Technology, Wuhan 430074, China.}
\email{linjie\_zhao@hust.edu.cn}

\thanks{\textbf{Acknowledgments.}  The project is supported by the National Natural Science Foundation of China with grant numbers 12401168 and 12371142, and the Fundamental Research Funds for the Central Universities in China.}

\begin{document}
	
	\begin{abstract}
		We study the weakly asymmetric simple exclusion process in one dimension. We prove  sample path moderate deviation principles for the current and the tagged particle  when the process starts from one of its stationary measures. We simplify the proof in our previous works \cite{xue2023moderate,xue2024sample}, where the same problem was investigated in the symmetric simple exclusion process. 
	\end{abstract}

\maketitle

\section{Introduction}

The exclusion process is one of the most studied interacting particle systems \cite{liggettips,liggett99}. In the dynamics, particles perform random walks on some graph subjected to the exclusion rule, that is, there is at most one particle at each site and jumping to occupied sites is suppressed.  It has been a long standing problem to study the behavior of a typical particle, usually called the tagged particle in the literature. The main challenge is that the tagged particle itself is not Markovian. By studying the environment process of the tagged particle, which is Markovian, much progress has been made, such as  law of large numbers \cite{saada1987tagged,rezakhanlou1994evolution}, central limit  theorems in equilibrium \cite{kipnis1986central,varadhan1995self,Kipnis86,sethuraman2000diffusive,sethuraman2006diffusive,zhao2023motion} and in nonequilibrium \cite{jara2008quenched,jara2006nonequilibrium,jara2009nonequilibrium}. Recently, large deviations \cite{sethuraman2013large,sethuraman2023atypical} and moderate deviations \cite{xue2023moderate,xue2024sample} for the current and the tagged particle in one dimension were also investigated. We refer the readers to \cite{komorowski2012fluctuations} and the above references for an excellent review on the related literature.

The aim of this paper is to simplify the proof of our previous works on sample path moderate deviation principles for the current and the tagged particle in the symmetric simple exclusion process in one dimension \cite{xue2023moderate,xue2024sample}.  To make our results more general, we study the weakly asymmetric simple exclusion process in one dimension. We show that when the process starts from one of its stationary measures, the current and the tagged particle process satisfy the moderate deviation principles under the correct time scaling.

The idea of the proof is as follows. Since particles cannot take over each other in the model, one can relate the current and the tagged particle to the empirical measure of the process. In \cite{zhao2024moderate}, the second author proved moderate deviation principles for the empirical measure of the process. Then, by contraction principle, the current and the tagged particle also satisfy moderate deviation principles. However, the rate functions for the current and the tagged particle are expressed as a  variational formula through this approach. By long calculations, the same authors solved it by using the Fourier approach in \cite{xue2023moderate,xue2024sample}. In this paper, we present a different approach, mainly relying on the contraction principle, and simplifies the previous proof significantly. The main observation is that the rate functions for the density fluctuation fields and for the limit of the density fluctuation fields are the same, which allows us to use the contraction principle back and forth. It seems that the approach developed here can also simplify the proof in \cite{zhao2024moderate}, where sample path moderate deviation principles for the occupation time of the exclusion process were proved in one dimension. 

The paper is organized as follows. In Section \ref{sec: model}, we introduce the process and the main results. We state moderate deviation principles for the empirical measure of the process in Section \ref{sec: density flieds}. Finally, the proof for moderate deviation principles of the current and the tagged particle is presented in Section \ref{sec: proof}.

\section{Model and Results}\label{sec: model}

We study the weakly asymmetric simple exclusion process (WASEP) on the one-dimensional integer lattice $\Z$. The state space of the process is $\Omega:= \{0,1\}^\Z$. For a configuration $\eta \in \Omega$, $\eta_x = 1$ if and only if there is one particle at site $x \in \Z$. Let $\alpha ,\beta,\gamma\geq 0$ be three parameters. Throughout the article, we shall take
\[\gamma = \min \{1+\beta, 2\}.\]
The infinitesimal generator of the WASEP is given by 
\[L_n = n^{\gamma} (L_s + \alpha n^{-\beta}L_a).\] 
Above, $L_s$ is associated to the dynamics of the symmetric simple exclusion process (SSEP), which acts on local functions $f: \Omega \rightarrow \R$ as
\[L_s f (\eta) = \frac{1}{2} \sum_{x \in \Z} \big[ f(\eta^{x,x+1} )- f(\eta)\big],\]
and $L_a$ is associated to the dynamics of the totally asymmetric simple exclusion process (TASEP),
\[L_a f (\eta) = \sum_{x \in \Z} \eta_x (1-\eta_{x+1})\big[ f(\eta^{x,x+1}) - f(\eta)\big],\]
where $\eta^{x,y}$ is the configuration obtained from $\eta$ by swapping the values of $\eta_x$ and $\eta_y$,
\[(\eta^{x,y})_{z} = \begin{cases}
 \eta_x, \quad \text{if } z = y,\\
  \eta_y, \quad \text{if } z = x,\\
   \eta_z, \quad \text{if } z \neq x, y.\\
\end{cases}\]
Note that if $\alpha = 0$ and $\gamma = 2$, then we obtain the  SSEP  in one dimension under the diffusive scaling; if $\beta = 0$ and $\gamma=1$, the we get the asymmetric simple exclusion process (ASEP) in one dimension under the hyperbolic scaling.

For $\rho \in [0,1]$, let $\nu_{\rho}$ be the product measure on the configuration space $\Omega$ with particle density $\rho$,
\[\nu_\rho (\eta_x = 1) = \rho, \quad  x \in \Z.\]
It is well known that the measure $\nu_\rho$ is reversible for the generator $L_s$ and is invariant for $L_a$, see \cite{liggettips} for example.

For any probability measure $\mu$ on $\Omega$, denote by $\Pbb_{\mu} \equiv \Pbb_{\mu}^n$ the probability measure on the path space $D(\R_+,\Omega)$ endowed with the Skorokhod topology corresponding to the law of the process $\eta(t)$ with generator $L_n$ and with initial distribution $\mu$. Let $\E_\mu \equiv \E_\mu^n$ be the corresponding expectation. 

We are interested in the current of the WASEP when it starts from its invariant measure $\nu_\rho$. For $x \in \Z$, the current $J^n_{x,x+1} (t)$ is defined as the number of particles jumping from $x$ to $x+1$ up to time $t$ minus the number of particles jumping from $x+1$ to $x$ up to time $t$. 

We shall also investigate the long-time behavior of the tagged particle. To define it, let $\nu_\rho^*$ be the measure $\nu_\rho$ conditioned on having a particle at the origin,
\[\nu_\rho^* (\cdot) = \nu_\rho (\cdot | \eta_0 = 1).\]
We call the particle initially at the origin the tagged particle. Under $\nu_\rho^*$, let $X^n (t)$ be the position of the tagged particle at time $t$. It is also well known that the measure $\nu_\rho^*$  is invariant for the environment process, defined as $\eta_{x + X^n (t)}(t)$ for $x \in \Z$,  as seen from the tagged particle, see \cite{liggettips} for example.

\subsection{Invariance principles} In this subsection, we state invariance principles for the current and the tagged particle.  In the rest of this article, we fix a time horizon $T > 0$. For any $t,s>0$, define the variance function as
\begin{equation}\label{a t s}
    a (t,s) = \begin{cases}
        \chi(\rho)\alpha|1-2\rho|\min\{t,s\} \quad &\text{if }\; \beta < 1,\\
        \chi(\rho)\left(f(t)+f(s)-f(|t-s|)\right) \quad &\text{if }\; \beta = 1,\\
        \chi(\rho)(\sqrt{t}+\sqrt{s}-\sqrt{|t-s|}) / \sqrt{2\pi}  \quad &\text{if }\; \beta > 1,
    \end{cases}
\end{equation}
 where $\chi (\rho) = \rho (1-\rho)$, and 
\[
f(t)=\frac{\alpha|1-2\rho|t}{2}+\mathbb{E}\left[\left(B_t-\alpha|1-2\rho|t\right)_+\right].
\]
Here, $\{B_t\}_{t\geq 0}$ is the standard one dimensional Brownian motion starting from the origin. For $r \in \R$, $r_+ := \max \{r,0\}$.  Note that $a (t,s) = 0$ when $\rho = 1/2$ and $\beta < 1$. Define
\[\bar{J}^n_{-1,0} (t) := J^n_{-1,0} (t) - t \alpha n^{\gamma-\beta}\chi (\rho), \quad \bar{X}^n (t) = X^n (t) - t \alpha n^{\gamma-\beta} (1-\rho).\]

By following \cite{gonccalves2008central} line by line, where fluctuations for the current and the tagged particle in the ASEP were investigated, one can prove the following result directly. The main idea is to relate the current and the tagged particle to the density fluctuation field of the process. For this reason, we omit the proof here.

\begin{proposition}\label{prop clt current tagged particle} Assume that $\rho \neq 1/2$ or $\beta \geq 1$. Then, as $n \rightarrow \infty$,  the sequence of processes $\{\bar{J}^n_{-1,0} (t)/\sqrt{n}, 0 \leq t \leq T \}_{n \geq 1}$ under $\mathbb{P}_{\nu_\rho}$ (respectively $\{\bar{X}^n (t)/\sqrt{n}, 0 \leq t \leq T \}_{n \geq 1}$ under $\mathbb{P}_{\nu_\rho^*}$) converges in the space $D([0,T],\R)$ to a Gaussian process with covariance $a(t,s)$ (respectively $a(t,s)/\rho^2$).
\end{proposition}

\begin{remark}
Note that when $\beta < 1$ and $\rho \neq 1/2$, the limiting process  is a  Brownian motion; when $\beta > 1$, the limit is a  fractional Brownian motion with Hurst parameter $1/4$.
\end{remark}

\begin{remark}
If $\beta=0$ and $\gamma = 1$, that is, for the ASEP, it was shown in \cite{balazs2010order} that  the variance of the current has order $n^{2/3}$ if $\rho = 1/2$.  For the TASEP, the one-point limit was proved in \cite{Johansson00} to be the Tracy-Widom distribution.
\end{remark}

\subsection{Moderate deviations}  In this subsection, we study moderate deviations for the current and the tagged particle.  
Let $\{a_n\}_{n \geq 1}$ be a sequence of real numbers such that \[\sqrt{n \log n} \ll a_n \ll n.\]  
Let $\mathbf{0}$ be the trajectory equals to zero at any time $0 \leq t \leq T$. Then, as $n \rightarrow \infty$,
\[\{\bar{J}^n_{-1,0} (t)/a_n, 0 \leq t \leq T \} \Rightarrow \mathbf{0}, \quad \{\bar{X}^n (t)/a_n, 0 \leq t \leq T \} \Rightarrow \mathbf{0}.\]
Moderate deviations concern the rate of the above convergence.

To introduce the MDP rate functions for the current and the tagged particle, we first recall the following  representation for $\{B_t^{1/4}, t \geq 0\}$ the fractional Brownian motion with Hurst parameter $1/4$ (see \cite{Decreu1999FBM} for example),
\[B_t^{1/4} = \int_0^t \mathcal{K} (t,s) d B_s, \quad t \geq 0,\]
where the kernel $ \mathcal{K} (t,s)$ is defined as
\[
\mathcal{K}(t,s)=\frac{(t-s)^{-1/4}}{\sqrt{V}\Gamma(3/4)}F(1/4, -1/4, 3/4, 1-\tfrac{t}{s}), \quad 0\leq s<t.
\]
Above, 
\[
V=\frac{8\Gamma(3/2)\cos(\pi/4)}{\pi}, \quad 
F(\alpha,\beta,\gamma,z)=\sum_{k=0}^{+\infty}\frac{(\alpha)_k(\beta)_k}{(\gamma)_kk!}z^k
\]
with $(a)_k := \Gamma(a+k)/\Gamma(a)$. 

If $\beta < 1$, we define $\mathcal{H}_\beta = C^1 ([0,T])$, the space of continuously differentiable functions on $[0,T]$. For $h \in \mathcal{H}_\beta$, denote by $\dot{h}$ the usual derivative of $h$.  In this case, the rate function for the current is defined as
	\[\mathcal{J}^\beta (h):= 
\frac{1}{2 \chi(\rho) \alpha |1-2\rho|}  \|\dot{h}\|_{L^2 ([0,T])}^2 \]
if $h \in \mathcal{H}_\beta$, and $= + \infty$ otherwise.

If $\beta > 1$, define $\mathcal{H}_\beta$ as the family of functions $h: [0,T] \rightarrow \R$ satisfying that there exists $\dot{h} \in L^2 ([0,T])$ such that
\[h (t) = \int_0^t \mathcal{K}(t,s) \dot{h} (s) ds, \quad \forall 0 \leq t \leq T.\]
In this case, the rate function for the current is defined as
	\[\mathcal{J}^\beta (h):= 	\frac{\sqrt{\pi}}{2  \sqrt{2} \chi(\rho)}  \|\dot{h}\|_{L^2 ([0,T])}^2.\]
if $h \in \mathcal{H}_\beta$, and $= + \infty$ otherwise. 

Finally, the rate function for the tagged particle is defined as
 \[\mathcal{X}^\beta (\cdot) := \rho^2 \mathcal{J}^\beta (\cdot).\]

Now, we state the main result of this article. 

\begin{theorem}\label{thm mdp current tagged particle} Assume that $\beta < 1, \rho \neq 1/2$ or $\beta > 1$. Then, the sequence  of processes $\{\bar{J}^n_{-1,0} (t)/a_n, 0 \leq t \leq T \}_{n \geq 1}$ under $\mathbb{P}_{\nu_{\rho}}$  (respectively $\{\bar{X}^n (t)/a_n, 0 \leq t \leq T \}_{n \geq 1}$ under $\mathbb{P}_{\nu_{\rho}^*}$) satisfies the moderate deviation principles in the space $D([0,T], \R)$ with decay rate $a_n^2/n$ and with rate function $\mathcal{J}^\beta$ (respectively with rate function $\mathcal{X}^\beta$). 	
	
Precisely speaking, for any closed set $\mathcal{C} \subset D([0,T], \R)$ and for any open set $\mathcal{O} \subset D([0,T],\R)$,
	\begin{align*}
		\limsup_{n \rightarrow \infty} \frac{n}{a_n^2} \log \mathbb{P}_{\nu_{\rho}} \Big(\{\bar{J}^n_{-1,0} (t)/a_n, 0 \leq t \leq T \} \in \mathcal{C}\Big) \leq - \inf_{h \in \mathcal{C}} \mathcal{J}^\beta (h),\\
		\liminf_{n \rightarrow \infty} \frac{n}{a_n^2} \log \mathbb{P}_{\nu_{\rho}} \Big(\{\bar{J}^n_{-1,0} (t)/a_n, 0 \leq t \leq T \} \in \mathcal{O}\Big) \geq - \inf_{h \in \mathcal{O}} \mathcal{J}^\beta (h).
	\end{align*}

	The precise meaning of the moderate deviation principles for the tagged particle is similar.
\end{theorem}

\begin{remark}\label{rmk: beta equals 1}
For $\beta = 1$,  the moderate deviation principles still hold for the current and the tagged particle.  The decay rate is $a_n^2/n$ as before. The rate functions are implicitly given by the following variational formulas: for $h \in D([0,T],\R)$,
\begin{align*}
	\mathcal{J}^\beta (h) &=\inf \Big\{\frac{1}{2} \mathbf{h}^T A^{-1} \mathbf{h}:  m \geq 1, 0 \leq t_1 < t_2 < \ldots < t_m \leq T, t_1,\ldots,t_m \in \Delta_c (h)\Big\},\\
	\mathcal{X}^\beta (h) &= \rho^2 \mathcal{J}^\beta (h),
\end{align*}
where $\mathbf{h} = (h(t_1), \ldots, h(t_m))^T \in \R^m$, $A = (a(t_i,t_j))_{1 \leq i,j \leq m}$ with $a(\cdot,\cdot)$ defined in \eqref{a t s}, and $\Delta_c (h)$ is the set of continuous points of $h$.  However, we are not aware of how to solve the above infimum explicitly so far.
\end{remark}

\begin{remark}
The above theorem should hold in the regime $\sqrt{n} \ll a_n \ll n$.  We need the technical assumption $a_n \gg \sqrt{n \log n}$ in order to prove the exponential tightness of the current and the tagged particle processes. However, we are not aware of how to remove this technical assumption.
\end{remark}

\section{Density fluctuation fields}\label{sec: density flieds}

In this section, we study the density fluctuation fields of the process. The  density fluctuation field $\Ycal^n_t$ of the WASEP, which acts on Schwartz functions $H \in \Scal (\R)$, is defined  as
\[\Ycal^n_t (H) = \frac{1}{\sqrt{n}} \sum_{x \in \Z} \Bar{\eta}_x (t)H(\tfrac{x}{n}),\]
where $\Bar{\eta}_x = \eta_x - \rho$.

The following result concerns the stationary fluctuations for the WASEP.  Its proof is based on the  martingale approach, which is standard in the theory of hydrodynamic limits, see \cite[Chapter 11]{klscaling}  for example.  For this reason, we omit the proof here.

\begin{proposition}\label{prop clt filed}
	Under $\Pbb_{\nu_\rho}$, the sequence of processes $\{\Ycal^n_t, 0 \leq t \leq T\}_{n \geq 1}$ converges in distribution in the space $D([0,T], \Scal^\prime (\R))$, as $n \rightarrow \infty$, to the process $\{\Ycal_t, 0 \leq t \leq T\}$, which is the unique solution to the following SPDE:
	\begin{enumerate}[(i)]
		\item if $\beta > 1$, then
		\[\partial_t \Ycal_t = \frac{1}{2} \Delta \Ycal_t + \sqrt{\chi (\rho)} \nabla d \mathcal{W}_t;\]
		\item if $\beta = 1$, then
		\[\partial_t \Ycal_t = \frac{1}{2} \Delta \Ycal_t - \alpha(1-2\rho) \nabla \Ycal_t+ \sqrt{\chi (\rho)} \nabla d \mathcal{W}_t;\]
		\item if $\beta < 1$, then
		\[\partial_t \Ycal_t  = - \alpha(1-2\rho) \nabla \Ycal_t.\]
	\end{enumerate}
	Moreover, the  initial distribution $\mathcal{Y}_0$ of the above processes satisfies that  the distribution of $\mathcal{Y}_0 (H)$ is normal  with mean $0$ and variance $\chi(\rho)\langle H, H\rangle$ for any $H\in \Scal(\R)$. 
\end{proposition}

Next, we study large deviations for the limiting process $\mathcal{Y}_t$ in Proposition \ref{prop clt filed}. We first introduce the  rate function.  For  $\beta>0$ and $\mu\in D ([0, T], \Scal^\prime(\R))$, define
\begin{equation}\label{equ mdp rate function of the fluctuation field}
    \mathcal{Q}^\beta(\mu)=\mathcal{Q}_{\rm ini}(\mu_0)+\mathcal{Q}_{\rm dyn}^\beta(\mu),
\end{equation}
where $\mathcal{Q}_{\rm ini}$ corresponds to the deviation from the initial state, and $\mathcal{Q}_{\rm dyn}^\beta$ comes from the evolution of the dynamics. The initial rate function $\mathcal{Q}_{\rm ini}$ is defined as 
\[
\mathcal{Q}_{\rm ini}(\mu_0)=\sup_{h\in \Scal(\R)}\left\{\mu_0 (h)-\frac{\chi(\rho)}{2}\int_\mathbb{R}h^2(u)du\right\}, \quad \mu_0 \in \Scal^\prime(\R)
\]
To define $\mathcal{Q}_{\rm dyn}^\beta$, we introduce the following positive-definite quadratic form:  for given $T>0$ and $G,H\in D ([0, T], \Scal(\R))$,  define
\[
[G, H]:= \int_0^T\langle \nabla G_t, \nabla H_t \rangle dt,
\]
where $\langle \cdot,\cdot \rangle$ is the usual inner product in $L^2 (\R)$. Let $\mu\in D ([0, T], \Scal^\prime(\R))$.
\begin{itemize}
	\item For $\beta<1$,  define
	\begin{equation}\label{equ dynamic mdp rate sub}
		\mathcal{Q}_{\rm dyn}^\beta(\mu)=\sup_{H\in C_c^{1,+\infty}([0, T]\times \mathbb{R})}\left\{\mu_T(H_T)-\mu_0(H_0)-\int_0^T\mu_t\left((\partial_t+\alpha(1-2\rho)\nabla)H_t\right)dt\right\}.    
	\end{equation}
	\item For $\beta>1$, define
	\begin{equation}\label{equ dynamic mdp rate sup}
		\mathcal{Q}_{\rm dyn}^\beta(\mu)
		=\sup_{H\in C_c^{1,+\infty}([0, T]\times \mathbb{R})}\left\{\mu_T(H_T)-\mu_0(H_0)-\int_0^T\mu_t\left((\partial_t+\frac{1}{2}\Delta)H_t\right)dt-\frac{\chi(\rho)}{2}[H, H]\right\}.  
	\end{equation}
	\item For $\beta=1$, define
	\begin{equation}\label{equ dynamic mdp rate cri}
		\mathcal{Q}_{\rm dyn}^1(\mu)
		=\sup_{H\in C_c^{1,+\infty}([0, T]\times \mathbb{R})}\left\{\mu_T(H_T)-\mu_0(H_0)-\int_0^T\mu_t\left((\partial_t+\mathcal{P})H_t\right)dt-\frac{\chi(\rho)}{2}[H, H]\right\}, 
	\end{equation}
	where $\mathcal{P}=\frac{1}{2}\Delta+\alpha(1-2\rho)\nabla$. 
\end{itemize}
Below, we write $\mathcal{Q}_{\rm dyn}^\beta, \mathcal{Q}^\beta$ as $\mathcal{Q}_{{\rm dyn}, T}^\beta, \mathcal{Q}_T^\beta$ respectively when we need to emphasize the dependence of the rate functions on the time horizon $T$. 

For later use, we give some properties of the above rate functions.  For any $G,H \in C_c^{1, +\infty}([0, T]\times \R)$, define the equivalence relation  
\[G\sim H \quad \text{if and only if} \quad [G-H, G-H]=0.\] 
We denote by $\mathbb{H}$ the completion of $C_c^{1, +\infty}([0, T]\times \R)/\sim$ under the inner product $[\cdot, \cdot]$. 

\begin{lemma}\label{lemma finite rate} 
$(1)$ If $\mu_0 \in \Scal^\prime(\R)$ satisfies that $\mathcal{Q}_{\rm ini}(\mu_0)<+\infty$, then there exists $\varphi \in L^2(\R)$ such that
\[
\mu_0 (H) = \langle \varphi, H\rangle, \; \forall H \in L^2 (\R), \quad \text{and} \quad \mathcal{Q}_{\rm ini}(\mu_0)=\frac{1}{2\chi(\rho)} \|\varphi\|_{L^2 (\R)}^2.
\]
$(2)$ If $\mu\in D ([0, T], \Scal^\prime(\R))$ satisfies that $\mathcal{Q}^\beta(\mu)<+\infty$, then $\mu\in C([0, T], L^2(\R))$. Moreover, we have the following characterizations of  $\mu$ and the rate function.

\begin{itemize}
	\item If  $\beta>1$, then there exists $G\in \mathbb{H}$ such that $\mu$ is the unique weak solution to the PDE
	\[\begin{cases}
	\partial_t \mu(t,u)=\frac{1}{2}\Delta\mu(t,u)-\Delta G(t,u), \text{~}t> 0, u\in \mathbb{R},\\
	\mu(0,u) = \varphi (u), \; u \in \R,
	\end{cases}\]
where $\varphi$ is identified in $(1)$. Moreover, $\mathcal{Q}^\beta_{\rm dyn}(\mu)=[G,G] / (2 \chi(\rho))$.
\item If  $\beta=1$, then there exists $G\in \mathbb{H}$ such that $\mu$ is the unique weak solution to the PDE
\[\begin{cases}
	\partial_t \mu(t,u)=\mathcal{P}^* \mu(t,u)-\Delta G(t,u), \text{~}t\geq 0, u\in \mathbb{R},\\
	\mu(0,u) = \varphi (u), \; u \in \R.
\end{cases}\]
Here, $\mathcal{P}^*$ is the adjoint of $\mathcal{P}$ in $L^2 (\R)$. Moreover, $\mathcal{Q}^\beta_{\rm dyn}(\mu)=[G,G] / (2 \chi(\rho))$.
\item If  $\beta< 1$, then 
\[
\mu(t,u)=\varphi( u-\alpha(1-2\rho)t), \text{~}t\geq 0, u\in \mathbb{R}.
\]
Moreover, $\mathcal{Q}^\beta_{\rm dyn}(\mu)=0$.
\end{itemize}
\end{lemma}

\begin{proof}[Proof of Lemma \ref{lemma finite rate}]
We only prove the case $\beta<1$ in (2) since the remaining statements follow from  Riesz representation theorem directly, see \cite{gao2003moderate, kipnis1989hydrodynamics} for example. If there exists $H\in C_c^{1,+\infty}([0, T]\times \mathbb{R})$ such that
\[
\mu_T(H_T)-\mu_0(H_0)-\int_0^T\mu_t\left((\partial_t+\alpha(1-2\rho)\nabla)H_t\right)dt\neq 0,
\]
then
\[
\mu_T(aH_T)-\mu_0(aH_0)-\int_0^T\mu_t\left((\partial_t+\alpha(1-2\rho)\nabla)(aH_t)\right)dt\rightarrow+\infty
\]
as $a\rightarrow+\infty$ or $a\rightarrow -\infty$, and hence $\mathcal{Q}^\beta_{\rm dyn}(\mu)=+\infty$. Consequently, if $\mathcal{Q}^\beta(\mu)<+\infty$, we must have
\[
\mu_T(H_T)-\mu_0(H_0)-\int_0^T\mu_t\left((\partial_t+\alpha(1-2\rho)\nabla)H_t\right)dt=0
\]
for any $H\in C_c^{1,+\infty}([0, T]\times \mathbb{R})$ and hence $\mathcal{Q}_{\rm dyn}^\beta(\mu)=0$.

Take the test function $H$ with the form $H(t,u)=b_th(u)$ for some $h\in C_c^\infty(\R), b\in C^1 ([0, T])$, then
\[
b_T\mu_T(h)-b_0\mu_0(h)-\int_0^Tb^\prime_t\mu_t(h)dt=\int_0^Tb_t\alpha(1-2\rho)\mu_t(\nabla h)dt.
\]
Since $b$ is arbitrary, we have
\[
\frac{d}{dt}\mu_t(h)=\alpha(1-2\rho)\mu_t(\nabla h).
\]
For $s\in \R$, let $S_{s}$ be the translation operator:  $S_{s}h(u)=h(u+\alpha(1-2\rho)s)$.  Then, for any $t\geq 0$,
\[
\frac{d}{dt}\mu_t(S_{-t}h)=\alpha(1-2\rho)\mu_t(\nabla S_{-t}h)-\alpha(1-2\rho)\mu_t(\nabla S_{-t}h)=0.
\]
 Hence,
\[
\mu_t(h)=\mu_0 (S_th)=\int_{\R}\varphi(u-\alpha(1-2\rho)t)h(u)du
\]
for any $h\in C_c^\infty(\R)$, thus concluding the proof. 
\end{proof}

Since the process $\mathcal{Y}_t$ is Gaussian, the following result is straightforward and thus the proof is omitted.  

\begin{lemma}\label{lem MPD of the fluctuation limit}
The sequence of processes $\left\{\frac{\sqrt{n}}{a_n}\mathcal{Y}_t:~0\leq t\leq T\right\}_{n\geq 1}$ satisfies the large deviation principles with decay rate $a_n^2/n$ and with rate function $\mathcal{Q}^\beta$. 
\end{lemma}

Finally, we study moderate deviations for the  density fluctuation field. The rescaled density fluctuation field $\mu^n=\{\mu^n_t\}_{0\leq t\leq T}$ of the process is defined as  
\begin{equation}\label{equ intial rate funtion}
	\mu^n_t(du)=\frac{1}{a_n}\sum_{x\in \mathbb{Z}}\overline{\eta}_x(t)\delta_{x/n}(du)
\end{equation}
where $\delta_a(du)$ is the Kronecker Dirac measure concentrated at the point $a$.

\begin{proposition}\label{prop MPD of the scaled density field}
	The rescaled density field $\{\mu^n_t, 0 \leq t \leq T\}_{n\geq 1}$ satisfies the moderate deviation principles in the space $D([0,T],\mathcal{S}^\prime (\R))$ with decay rate $a_n^2 / n$ and with rate function $\mathcal{Q}^\beta$.    
\end{proposition}

 Its proof is similar to \cite{zhao2024moderate}, where the MDP was proved for the process with generator $n^2 (L_s + \alpha n^{-\beta} L_a)$. For that reason, we omit the proof here.

\section{Proof of Theorem \ref{thm mdp current tagged particle}}\label{sec: proof}

In this section, we prove moderate deviation principles for the current and the tagged particle. In Subsection \ref{subsec: current}, we prove sample path moderate deviation principles for the current by assuming the exponential tightness and finite dimensional moderate deviation principles of the current. The exponential tightness for the current are proved in Subsection \ref{subsec: pf lemma 4.1}, and the finite dimensional moderate deviation principles are proved in Subsection \ref{subsec:pf lemma 4.2}.  Finally, sample path moderate deviation principles for the tagged particle are outlined in Subsection \ref{subsec:tagged}.

\subsection{MDP for the current}\label{subsec: current} To prove the sample path MDP for the current, we only need to prove the exponential tightness and finite-dimensional MDP for the process $\{\bar{J}^n_{-1,0} (t)/a_n, 0 \leq t \leq T\}_{n \geq 1}$, which are summarized in the following two lemmas. 

\begin{lemma}\label{lem: exp tight current}
The sequence of processes $\{\bar{J}^n_{-1,0}(t) / a_n, 0 \leq t \leq T\}_{n \geq 1}$ is exponentially tight.
\end{lemma} 

\begin{lemma}\label{lem: finiteD mdp current}
For any $m \geq 1$ and for any $0 \leq t_1 < t_2 < \ldots < t_m \leq T$, the sequence of random vectors $\{\bar{J}^n_{-1,0} (t_i)/a_n, 1 \leq i \leq m\}_{n \geq 1}$ satisfies the MDP with decay rate $a_n^2 / n$ and with rate function $\mathcal{J}^{\beta,m} \equiv \mathcal{J}^\beta_{\{t_i\}_{i=1}^m}$, where  for $\mathbf{r} =(r_1,\ldots, r_m)^T\in \mathbb{R}^m$,
\[\mathcal{J}^{\beta,m}  (\mathbf{r}) = \frac{1}{2} \mathbf{r}^T A^{-1} \mathbf{r}.\]
Here, $A = A_{\{t_i\}_{i=1}^m} = (a(t_i,t_j))_{1 \leq i,j \leq m}$ is the $m \times m$ matrix with $a(\cdot,\cdot)$ defined in \eqref{a t s}. 
\end{lemma}

Now, we prove the sample path MDP for the current by using the above two lemmas.

\begin{proof}[Proof of Theorem \ref{thm mdp current tagged particle} for the current]
By Lemmas \ref{lem: exp tight current}, \ref{lem: finiteD mdp current} and \cite[Theorem 4.28]{Feng2006LDP}, the sequence of processes $\{\bar{J}^n_{-1,0} (t)/a_n, 0 \leq t \leq T\}_{n \geq 1}$ satisfies the MDP with decay rate $a_n^2/n$ and with rate function
\[\mathcal{J}^\beta (r(\cdot)):=\inf \Big\{\frac{1}{2} \mathbf{r}^T A^{-1} \mathbf{r}:  m \geq 1, 0 \leq t_1 < t_2 < \ldots < t_m \leq T, t_1,\ldots,t_m \in \Delta_c (r(\cdot))\Big\},\]
where $\Delta_c (r(\cdot))$ is the set of continuous points of the function $r: [0,T] \rightarrow \R$, $\mathbf{r} = (r(t_1),\ldots,r(t_m))^T \in \R^m$ and $A=  (a(t_i,t_j))_{1 \leq i,j \leq m}$. Moreover, since $a(t,s)$ is the covariance function of the Brownian motion if $\beta < 1$, and that of the fractional Brownian motion with Hurst parameter $1/4$ if $\beta > 1$, using \cite[Theorem 4.28]{Feng2006LDP} again, the last infimum equals a constant multiple the LDP rate function of the Brownian motion (respectively the fractional Brownian motion with Hurst parameter $1/4$ ) if $\beta < 1$ (respectively if $\beta > 1$). Precisely,
\[\mathcal{J}^\beta (r(\cdot)) = \begin{cases}
\frac{1}{2 \chi(\rho) \alpha |1-2\rho|}  \|\dot{r}\|_{L^2 ([0,T])}^2 \quad &\text{if}\; \beta < 1,\\
\frac{\sqrt{\pi}}{2  \sqrt{2} \chi(\rho)}  \|\dot{r}\|_{L^2 ([0,T])}^2 \quad &\text{if}\; \beta > 1.\\
\end{cases}\]
This concludes the proof.
\end{proof}

\subsection{Proof of Lemma \ref{lem: exp tight current}}\label{subsec: pf lemma 4.1} We follow the proof of  \cite[Lemma 5.2]{xue2024sample},  where the exponential tightness was proved for the current in the SSEP. The proof in the previous paper used the stirring representation for the SSEP, which does not hold any more for the WASEP.  Despite that, most of the proof follows \cite[Lemma 5.2]{xue2024sample} line by line. Thus, we only sketch it here and underline the main differences.

It suffices to prove the following two estimates:
\begin{itemize}
	\item \begin{equation}\label{exp tight c1}
		\lim_{M \rightarrow \infty} \limsup_{n \rightarrow \infty} \frac{n}{a_n^2} \log \Pbb_{\nu_\rho} \Big( \sup_{0 \leq t \leq T} |\bar{J}^n_{-1,0} (t) | > a_n M\Big) = -\infty;
	\end{equation}
\item for any $\varepsilon > 0$,
\begin{equation}\label{exp tight c2}
\lim_{\delta \rightarrow 0} \limsup_{n \rightarrow \infty} \sup_{\tau \in \mathcal{T}_T} \frac{n}{a_n^2} \log \Pbb_{\nu_\rho} \Big(\sup_{0 \leq t \leq \delta} |\bar{J}^n_{-1,0} (t+\tau) - \bar{J}^n_{-1,0} (\tau)| >a_n \varepsilon\Big) = - \infty,
\end{equation}
where $\mathcal{T}_T$ is the family of all stopping times bounded by $T$. 
\end{itemize}

For any integer $l > 0$, define
\begin{equation}\label{G l}
	G_l(u)=(1-\tfrac{u}{l})\mathbf{1}_{\{0\leq u\leq l\}}, \quad u\in \mathbb{R}.
\end{equation}
Since the number of particles is conserved,
\begin{equation*}
	\eta_x (t) - \eta_x (0) = J^n_{x-1,x} (t) - J^n_{x,x+1} (t), \quad x \in \Z.
\end{equation*}
Multiplying by $G_l (x/n)$ on both hand sides, then summing over $x \in \Z$ and using the summation by parts formula, we have
\begin{equation}\label{current density filed relation}
	\begin{aligned}
		\sum_{x} G_l (\tfrac{x}{n}) [\eta_x (t) - \eta_x (0)] &= \sum_{x} G_l (\tfrac{x}{n}) [J^n_{x-1,x} (t) - J^n_{x,x+1} (t)]\\
		&= \sum_{x} [G_l (\tfrac{x+1}{n}) - G_l (\tfrac{x}{n})] J^n_{x,x+1} (t) \\
		&= J^n_{-1,0} (t) - \frac{1}{nl} \sum_{x=0}^{nl-1} J^n_{x,x+1} (t).
	\end{aligned}
\end{equation}
Note that the expectation of both hand sides in the last equation is zero with respect to $\Pbb_{\nu_\rho}$. Thus,
\begin{equation}\label{current decom}
\bar{J}^n_{-1,0} (t) / a_n = \langle \mu^n_t,G_l\rangle -\langle \mu^n_0,G_l  \rangle + \frac{1}{na_nl} \sum_{x=0}^{nl-1} \bar{J}^n_{x,x+1} (t).
\end{equation}

By following the proof of \cite[Lemma 5.1]{xue2024sample}, one can show that for any $\varepsilon > 0$,
\begin{equation}\label{super exp 1}
\limsup_{l\rightarrow+\infty} \limsup_{n \rightarrow \infty} \frac{n}{a_n^2} \log \Pbb_{\nu_\rho} \Big( \sup_{0 \leq t \leq T} \Big| \frac{1}{na_n l} \sum_{x=0}^{nl-1} \bar{J}^n_{x,x+1} (t) \Big| > \varepsilon \Big) = - \infty.
\end{equation}
In particular, the third term on the right hand side of \eqref{current decom} is exponentially tight as $n \rightarrow \infty, l \rightarrow \infty$.

To conclude the proof, we only need to show that the estimates in \eqref{exp tight c1} and \eqref{exp tight c2} are true with $\bar{J}^n_{-1,0} (t)$ replaced by $\langle \mu^n_t,G_l\rangle$. The problem is that the function $G_l$ does not belong to $\mathcal{S} (\R)$.  Thus, we need to approximate it by Schwartz functions. It is in this step that we need the technical assumption $a_n \gg \sqrt{n \log n}$. Let $\tilde{G}_l \in  \mathcal{S} (\R)$ satisfying
\[\|\tilde{G}_l  - G_l\|_{L^2(\R)} \leq Cl^{-1}, \quad {\rm supp} (\tilde{G}_l) \subset [-2\ell,2\ell].\]
By Proposition \ref{prop MPD of the scaled density field}, for any $l$, $\langle \mu^n_t,\tilde{G}_l\rangle$ is exponentially tight. Thus, we only need to  prove that,  for any $\varepsilon > 0$, 
 \begin{equation}\label{exp tight error}
\limsup_{n \rightarrow \infty} \frac{n}{a_n^2} \log \Pbb_{\nu_\rho} \Big( \sup_{0 \leq t \leq T} |\langle \mu^n_t,F_l\rangle | >  \varepsilon \Big) = -\infty,
\end{equation}
where  $F_l := \tilde{G}_l - G_l$. 

Let $t_i = i T / n^3$ for $0 \leq i \leq n^3$. Then, we bound the probability in \eqref{exp tight error} by 
\[\Pbb_{\nu_\rho} \Big( \sup_{0 \leq i \leq n^3} |\langle \mu^n_{t_i},F_l\rangle | >  \varepsilon/2 \Big) + \Pbb_{\nu_\rho} \Big( \sup_{0 \leq t \leq T} |\langle \mu^n_t,F_l\rangle | - \sup_{0 \leq i \leq n^3} |\langle \mu^n_{t_i},F_l\rangle |  >  \varepsilon/2 \Big).\]
By using the assumption $a_n \gg \sqrt{n \log n}$, we have
\[\limsup_{n \rightarrow \infty} \frac{n}{a_n^2} \log \Pbb_{\nu_\rho} \Big( \sup_{0 \leq i \leq n^3} |\langle \mu^n_{t_i},F_l\rangle | >  \varepsilon/2 \Big)  = - \infty,\]
see \cite[Proof of the first term in (5.16)]{xue2024sample} for details.  

To deal with the second term above, we need to modify the proof in \cite{xue2024sample} slightly since we cannot use the stirring representation for the SSEP here. Since the process is stationary,
\begin{align*}
	&\Pbb_{\nu_\rho} \Big( \sup_{0 \leq t \leq T} |\langle \mu^n_t,F_l\rangle | - \sup_{0 \leq i \leq n^3} |\langle \mu^n_{t_i},F_l\rangle |  >  \varepsilon/2 \Big) \\
	\leq& \Pbb_{\nu_\rho} \Big( \sup_{0 \leq i \leq n^3-1} \,\sup_{t_i \leq t \leq t_{i+1}} |\langle \mu^n_t - \mu^n_{t_i},F_l\rangle |  >  \varepsilon/2 \Big)\\
	\leq& n^3 \Pbb_{\nu_\rho} \Big( \sup_{0 \leq t \leq Tn^{-3}} |\langle \mu^n_t - \mu^n_{0},F_l\rangle |  >  \varepsilon/2 \Big).
\end{align*}
By basic coupling (see \cite{liggettips} for example), we can assume that there is a particle at the origin initially. We label this particle by $Y_0$ and then label the other particles from the left to the right in an increasing order. Let $Y_i (t)$ be the position of the particle with label $i \in \Z$ at time $t$. Since particles cannot take over each other, $Y_i (t) < Y_{i+1} (t)$ for any $i \in \Z$ and any $t \geq 0$. We rewrite
\[\langle \mu^n_t - \mu^n_{0},F_l \rangle = \frac{1}{a_n} \sum_{i \in \Z} \Big\{F_l \big(\tfrac{Y_i(t)}{n}\big) - F_l \big(\tfrac{Y_i(0)}{n}\big)\Big\}.\]
Next, we consider the two cases $|i| > 3 n l$ and $|i| \leq 3 n l$ respectively. For the first case, let $\xi$ be a Poisson random variable with parameter $\alpha T n^{-1}$. Then, $\sup_{0 \leq t \leq Tn^{-3}}|Y_i (t) - Y_i (0)|$ is stochastically bounded by $\xi$.  Note that $F_l$ is supported in $[-2l,2l]$ and $|Y_i (0)| \geq 3 n l$ for $|i| > 3 n l$. Thus, $F_l (Y_i (0) / n) = 0$ for  $|i| > 3 n l$.  By standard large deviation estimates, 
\begin{align*}
	&\Pbb_{\nu_\rho} \Big( \sup_{0 \leq t \leq Tn^{-3}} \big|\frac{1}{a_n} \sum_{|i| > 3 n l} \Big\{F_l \big(\tfrac{Y_i(t)}{n}\big) - F_l \big(\tfrac{Y_i(0)}{n}\big)  \big|  >  \varepsilon/2 \Big)\\
	\leq& \sum_{|i| > 3 n l} \Pbb_{\nu_\rho} \Big(  \sup_{0 \leq t \leq Tn^{-3}} |Y_i (t) - Y_i (0)| > |i| - 2 n l\Big)\\
	\leq& \sum_{|i| > 3 n l} \Pbb_{\nu_\rho} (\xi > |i| - 2 n l) \leq C e^{-Cnl}
\end{align*}
for some constant $C > 0$. To deal with the sum over  $|i| \leq 3 n l$, we introduce the event
\[A_n = \Big\{  \sup_{|i| \leq 3 n l} \sup_{0 \leq t \leq Tn^{-3}} |Y_i (t) - Y_i (0)| \leq a_n^{3/2}/n^{1/2}\Big\}.\]
Then,
\[\Pbb_{\nu_\rho}  (A_n^c) \leq (3nl+1) \Pbb_{\nu_\rho} (\xi > a_n^{3/2}n^{-1/2}) \leq C n l e^{-Ca_n^{3/2}n^{-1/2}}.\]
We claim that for $n$ large enough, 
\[A_n \cap \Big\{   \sup_{0 \leq t \leq Tn^{-3}} \big|\frac{1}{a_n} \sum_{|i| \leq 3 n l} \Big\{F_l \big(\tfrac{Y_i(t)}{n}\big) - F_l \big(\tfrac{Y_i(0)}{n}\big)  \big|  >  \varepsilon/2 \Big\} = \emptyset.\]
Adding up the above estimates,  for $n$ large enough,
\[	\Pbb_{\nu_\rho} \Big( \sup_{0 \leq t \leq T} |\langle \mu^n_t,F_l\rangle | - \sup_{0 \leq i \leq n^3} |\langle \mu^n_{t_i},F_l\rangle |  >  \varepsilon/2 \Big) \leq C n^3 e^{-C n l} + C n^4 l e^{-Ca_n^{3/2}n^{-1/2}}.\]
Since $a_n \ll n$,
\[\lim_{n \rightarrow \infty} \frac{n}{a_n^2} \log \Pbb_{\nu_\rho} \Big( \sup_{0 \leq t \leq T} |\langle \mu^n_t,F_l\rangle | - \sup_{0 \leq i \leq n^3} |\langle \mu^n_{t_i},F_l\rangle |  >  \varepsilon/2 \Big) = - \infty.\]

It remains to prove the claim. For any $t$, let $B_{n,t}$ be the set of labels $i \in [-3nl,3nl]$ such that $Y_i (0) < 0, Y_i (t) \geq 0$ or  $Y_i (0) > 0, Y_i (t) \leq 0$.  Since the orderings of the particles are preserved by the dynamics, on the event $A_n$ we have $|B_{n,t}| \leq C a_n^{3/2} n^{-1/2}$ for any $0 \leq t \leq T n^{-3}$. Thus, on the event $A_n$,
\[\sup_{0 \leq t \leq Tn^{-3}} \big|\frac{1}{a_n} \sum_{|i| \leq 3 n l, i \in B_{n,t}} \Big\{F_l \big(\tfrac{Y_i(t)}{n}\big) - F_l \big(\tfrac{Y_i(0)}{n}\big)  \Big\} \big| \leq C(l) a_n^{1/2} n^{-1/2}.\]
Note that $F(l)$ is only discontinuous at the origin. For $i \notin B_{n,t}$, by the piecewise smoothness of $F_l$, on the event $A_n$,
\[\sup_{0 \leq t \leq Tn^{-3}} \big|\frac{1}{a_n} \sum_{|i| \leq 3 n l, i \notin B_{n,t}} \Big\{F_l \big(\tfrac{Y_i(t)}{n}\big) - F_l \big(\tfrac{Y_i(0)}{n}\big)  \Big\} \big| \leq C(l) a_n^{1/2} n^{-1/2}. \]
Thus, on the event $A_n$, 
\[\sup_{0 \leq t \leq Tn^{-3}} \big|\frac{1}{a_n} \sum_{|i| \leq 3 n l} \Big\{F_l \big(\tfrac{Y_i(t)}{n}\big) - F_l \big(\tfrac{Y_i(0)}{n}\big)  \Big\} \big| \leq C(l) a_n^{1/2} n^{-1/2},\]
which cannot be larger than $\varepsilon / 2$ for $n$ large enough since $a_n \ll n$. This proves the claim. 

\subsection{Proof of Lemma \ref{lem: finiteD mdp current}}\label{subsec:pf lemma 4.2} Intuitively, since 
\[\bar{J}^n_{-1,0} (t) / a_n = \frac{1}{a_n} \sum_{x=0}^{+\infty} [\eta_x (t) - \eta_x (0)] = \langle \mu^n_t - \mu^n_0, \chi_{[0,\infty)}\rangle,\]
by contraction principle, $\{\bar{J}^n_{-1,0} (t_i)/a_n, 1 \leq i \leq m\}$ should satisfy the MDP with decay rate $a_n^2/n$ and with rate function
\[\inf\left\{\mathcal{Q}_{t_m}^\beta(\mu):~ \langle \mu_{t_i} -\mu_{0},  \chi_{[0,\infty) } \rangle=r_i\text{~for all~}1\leq i\leq m\right\}.\]
The main problem is that the indicator function $\chi_{[0,\infty)}$ is not a test function.  Moreover, we also need to solve the above variational problem explicitly. It was calculated in \cite{xue2024sample} by using Fourier analysis.  Based on the contraction principle, we present a different approach to solve the above variational problem, which simplifies the proof in \cite{xue2024sample} significantly. 

\subsubsection{The lower bound} We need to prove that for any open set $\mathcal{O} \subset \R^m$,
\begin{equation*}
	\liminf_{n \rightarrow \infty} \frac{n}{a_n^2} \log \mathbb{P}_{\nu_{\rho}} \Big( \{\bar{J}^n_{-1,0} (t_i)/a_n, 1 \leq i \leq m\} \in \mathcal{O} \Big) \geq - \inf_{\mathbf{r} \in \mathcal{O}} \frac{1}{2} \mathbf{r}^T A^{-1} \mathbf{r}.
\end{equation*}
Using the relation between the current and the fluctuation field in \eqref{current decom} and the super-exponential estimates in \eqref{super exp 1} and \eqref{exp tight error}, for any $\mathbf{r} \in  \mathcal{O}$ and any  $\varepsilon > 0$ such that $B(\mathbf{r},\varepsilon) \subset \mathcal{O}$,
 \begin{align*}
 	\liminf_{n \rightarrow \infty} \frac{n}{a_n^2} \log &\mathbb{P}_{\nu_{\rho}} \Big( \{\bar{J}^n_{-1,0} (t_i)/a_n, 1 \leq i \leq m\} \in \mathcal{O} \Big) \\
 	&\geq - \liminf_{l\rightarrow+\infty} \inf \{\mathcal{Q}^\beta_{t_m} (\mu): (\langle \mu_{t_i} - \mu_{0}, \tilde{G}_l\rangle, 1 \leq i \leq m ) \in B(\mathbf{r},\varepsilon)\}\\
 	&\geq - \liminf_{l\rightarrow+\infty} \inf \{\mathcal{Q}^\beta_{t_m} (\mu): \langle \mu_{t_i} - \mu_{0}, \tilde{G}_l \rangle=r_i\text{~for all~}1\leq i\leq m \},
 \end{align*}
where $B(\mathbf{r},\varepsilon)$ is the ball of radius $\varepsilon$ centered at the point $\mathbf{r}$, and $\tilde{G}_l$ was introduced before \eqref{exp tight error}. We refer the readers to \cite[Proof of (6.2)]{xue2024sample} for details of the above argument.  

To calculate the above infimum, we have the following result.

\begin{lemma}\label{lemma 5.1}
Let $G\in \Scal(\R)$ . For any $m \geq 1$,  any $0 \leq t_1 < t_2 < \ldots < t_m \leq T$ and any $\mathbf{r}=(r_1,\ldots, r_m)^T\in \mathbb{R}^m$,
\[		\inf\left\{\mathcal{Q}^\beta_{t_m} (\mu):~ \langle \mu_{t_i} - \mu_{0}, G \rangle=r_i\text{~for all~}1\leq i\leq m\right\}
=\sup_{\bm{\xi} \in \R^m}\left\{\bm{\xi}^T \mathbf{r}-\frac{1}{2}\bm{\xi}^T \Sigma \bm{\xi}\right\},\]
	where $\Sigma := \Sigma_{\{t_k\}_{k=1}^m}(G)$ is a $m\times m$ symmetric matrix such that
	\[
	\Sigma (i,j)={\rm Cov}\left(\mathcal{Y}_{t_i}(G)-\mathcal{Y}_{0}(G), \mathcal{Y}_{t_j}(G)-\mathcal{Y}_{0}(G)\right)
	\]
	for any $1\leq i,j\leq m$. In particular, when $\Sigma$ is invertible, 
	\[
\inf\left\{\mathcal{Q}^\beta_{t_m} (\mu):~ \langle \mu_{t_i} - \mu_{0}, G \rangle=r_i\text{~for all~}1\leq i\leq m\right\}=\frac{1}{2} \mathbf{r}^T\Sigma^{-1} \mathbf{r} .
	\]
\end{lemma}

\begin{proof}
	Since $\{\mathcal{Y}_t\}_{t\geq 0}$ is Gaussian,   the vector $\{\frac{\sqrt{n}}{a_n}\mathcal{Y}_{t_i}(G)-\frac{\sqrt{n}}{a_n}\mathcal{Y}_{0}(G):~1\leq i\leq m\}_{n\geq 1}$ satisfies the large deviation principles  with rate function 
	\[\sup_{\bm{\xi}\in \R^m}\left\{\bm{\xi}^T \mathbf{r}-\frac{1}{2} \bm{\xi}^T \Sigma \bm{\xi} \right\},\]
	see \cite{deuschel1989large} for example. Then, by Lemma \ref{lem MPD of the fluctuation limit} and the contraction principle, the first identity in Lemma \ref{lemma 5.1} holds. When $\Sigma$ is invertible, the second identity in Lemma \ref{lemma 5.1} follows directly from 
	Cauchy-Schwarz inequality. 
\end{proof}

By direct calculations, it is straightforward to prove that  for any $s,t > 0$,
\begin{equation}\label{cov converg}
\lim_{l\rightarrow+\infty}{\rm Cov}\left(\mathcal{Y}_t(\tilde{G}_l)-\mathcal{Y}_0(\tilde{G}_l), \mathcal{Y}_s(\tilde{G}_l)-\mathcal{Y}_0(\tilde{G}_l)\right) = a(t,s),
\end{equation}
where $a(\cdot,\cdot)$ was defined in \eqref{a t s}. Then, by Lemma \ref{lemma 5.1}, 
\begin{align*}
	&\liminf_{n \rightarrow \infty} \frac{n}{a_n^2} \log \mathbb{P}_{\nu_{\rho}} \Big( \{\bar{J}^n_{-1,0} (t_i)/a_n, 1 \leq i \leq m\} \in \mathcal{O} \Big) \\
	&\geq - \liminf_{l\rightarrow+\infty}  \frac{1}{2} \mathbf{r}^T \big(\Sigma_{\{t_k\}_{k=1}^m}(\tilde{G}_l)\big)^{-1} \mathbf{r} \\
	&=  - \frac{1}{2} \mathbf{r}^T A^{-1} \mathbf{r},
\end{align*}
where the matrix $A$ was defined in Lemma \ref{lem: finiteD mdp current}.  We conclude the proof of the lower bound by optimizing over $\mathbf{r} \in \mathcal{O}$. 

\subsubsection{The upper bound} By exponential tightness of the current, we only need to prove the upper bound for any compact set $\mathcal{K} \subset \R^m$. Following \cite[Proof of (6.1)]{xue2024sample}  line by line, one can show that, for any $\varepsilon > 0$,
\begin{multline*}
	\limsup_{n \rightarrow \infty} \frac{n}{a_n^2} \log \mathbb{P}_{\nu_{\rho}} \Big( \{\bar{J}^n_{-1,0} (t_i)/a_n, 1 \leq i \leq m\} \in \mathcal{K} \Big)\\
	\leq - \inf_{\mathbf{r} \in \mathcal{K}} \inf_{\mu}  \left\{\mathcal{Q}^\beta_{t_m} (\mu):~\int_0^{+\infty}[\mu(t_i,u)-\mu(0, u)]du=r_i\text{~for all~}1\leq i\leq m\right\}.
\end{multline*}

The following result bounds from below the first infimum in the last inequality, from which the upper bound follows immediately.

\begin{lemma}\label{lemma contraction principle of current}
For any $\beta\geq 0$ and any $\mathbf{r}=(r_1,\ldots, r_m)^T\in \mathbb{R}^m$, 
\[
\inf\left\{\mathcal{Q}^\beta_{t_m} (\mu):~\int_0^{+\infty} [\mu(t_i,u)-\mu(0, u)] du=r_i\text{~for all~}1\leq i\leq m\right\}\geq \frac{1}{2} \mathbf{r}^T A^{-1} \mathbf{r},
\]
where the matrix $A$ was defined in Lemma \ref{lem: finiteD mdp current}.
\end{lemma}

In order to prove Lemma \ref{lemma contraction principle of current}, we first calculate the macroscopic current across the origin.

\begin{lemma}\label{lemma 5.3}
	Assume $\mu\in D ([0, T], \Scal^\prime(\R))$ satisfies that $\mathcal{Q}_{T}^\beta (\mu)<+\infty$.   Let  $G = G(\mu)$ and $\varphi = \varphi (\mu)$ be identified  in Lemma \ref{lemma finite rate}.
	\begin{enumerate}[(i)]
		\item For $\beta<1$,
		\[
		\int_0^{+\infty}[\mu(t,u)-\varphi (u)]du=\int_{-\alpha(1-2\rho)t}^0 \varphi (u) du, \quad 0\leq t\leq T.
		\]
		\item For $\beta>1$,
		\begin{align*}
			\int_0^{+\infty}&[\mu(t,u)-\varphi (u)]du\\
			&=\int_{\R}\varphi (u) V_t(u)du-\int_0^t\left(\int_{\R}\partial_{uu}^2G (s,u)\mathbb{P}(B_{t-s}+u\geq 0)du\right)ds, \quad 0 \leq t \leq T,
		\end{align*}
		where $\{B_t\}_{t\geq 0}$ is the one dimensional standard Brownian motion starting from the origin and 
		\[
		V_t(u)=\mathbb{P}(B_t+u\geq 0)-\chi_{\{u\geq 0\}}=
		\begin{cases}
			\mathbb{P}(B_t\geq |u|) & \text{~if~}u<0,\\
			-\mathbb{P}(B_t\geq |u|) & \text{~if~}u\geq 0. 
		\end{cases}
		\]
		\item For $\beta=1$,
		\begin{align*}
			&\int_0^{+\infty} [\mu(t,u)-\varphi (u)] du=\int_{\R}\varphi (u)R_t(u)du\\
			&-\int_0^t\left(\int_{\R}\partial_{uu}^2 G (s,u)\mathbb{P}(B_{t-s}+u+\alpha(t-s)(1-2\rho)\geq 0)du\right)ds, \quad 0 \leq t \leq T,
		\end{align*}
		where 
		\[
		R_t(u)=\mathbb{P}(B_t\geq -u-\alpha(1-2\rho)t)-\chi_{\{u\geq 0\}}=
		\begin{cases}
			\mathbb{P}(B_t\geq -u-\alpha(1-2\rho)t) & \text{~if~}u<0,\\
			-\mathbb{P}(B_t\geq u+\alpha(1-2\rho)t) & \text{~if~}u\geq 0. 
		\end{cases}
		\] 
	\end{enumerate}
\end{lemma}

\begin{remark}
	Note that the integral in (ii) converges since
	\begin{align*}
		\int_{\R}\partial_{uu}^2G (s,u)\mathbb{P}(B_{t-s}+u\geq 0)du&=-\int_{\R}\partial_{u}G (s,u)\partial_u\mathbb{P}(B_{t-s}+u\geq 0)du\\
		&=-\int_{\R}\partial_{u}G (s,u)\frac{1}{\sqrt{2\pi(t-s)}}e^{-\frac{u^2}{2(t-s)}}du.
	\end{align*}
	Similarly, the integral in (iii) also converges. 
\end{remark}

	\begin{proof}
		
		We only deal with the  case $\beta = 1$ and the remaining two cases follow from the same analysis. Since $\mathcal{Q}^\beta_T (\mu) < + \infty$, by Lemma \ref{lemma finite rate},
		\[
		\partial_t\mu(t,u)=\frac{1}{2}\Delta\mu(t, u)-\alpha(1-2\rho)\nabla\mu(t, u)-\Delta G (t,u).
		\]
		Thus, for $0\leq t\leq T$,
		\[
		\mu(t,u)=S_{-t}T_t \varphi (u)-\int_0^tS_{-(t-s)}T_{t-s}\Delta G (s,u)ds,
		\]
		where $\{T_t\}_{t\geq 0}$ is the semigroup associated with the Laplacian $(1/2) \Delta$ and $\{S_s\}_{s \in \R}$ is the translation operator: $S_s h (u) = h (u+\alpha(1-2\rho)s)$. Then, 
		\[
		\int_0^{+\infty} [\mu(t,u)- \varphi (u)]du={\rm \uppercase\expandafter{\romannumeral1}}-{\rm \uppercase\expandafter{\romannumeral2}},
		\]
		where
		\begin{align*}
			{\rm \uppercase\expandafter{\romannumeral1}} &=\int_0^{+\infty}S_{-t}T_t\varphi (u)-\varphi (u)du,\\
			{\rm \uppercase\expandafter{\romannumeral2}} &=\int_0^{+\infty}\left(\int_0^tS_{-(t-s)}T_{t-s}\Delta G(s,u)ds\right)du.
		\end{align*}
		By direct calculations,
		\begin{align}\label{equ appendix C1}
			{\rm \uppercase\expandafter{\romannumeral1}}&=\lim_{M\rightarrow+\infty}\int_0^M\mathbb{E} [\varphi (B_t-\alpha(1-2\rho)t+u)] -\varphi (u)du\notag\\
			&=\lim_{M\rightarrow+\infty}\int_0^M\left(\int_{-\infty}^{+\infty}\frac{\varphi (v)}{\sqrt{2\pi t}}e^{-\frac{(v-u+\alpha(1-2\rho)t)^2}{2t}}dv\right)-\varphi (u)du \notag\\
			&=\lim_{M\rightarrow+\infty}\int_{-\infty}^{+\infty}\varphi (u)\left(\int_0^M\frac{1}{\sqrt{2\pi t}}e^{-\frac{(v-u-\alpha(1-2\rho)t)^2}{2t}}dv-\chi_{\{0\leq u\leq M\}}\right)du \\
			&=\lim_{M\rightarrow+\infty}\int_{-\infty}^{+\infty}\varphi (u) \left(\mathbb{P}(0\leq B_t+u+\alpha(1-2\rho)t\leq M)-\chi_{\{0\leq u\leq M\}}\right)du \notag\\
			&=\int_{-\infty}^{+\infty}\varphi (u)R_t(u)du \notag,
		\end{align}
		and 
		\begin{align}\label{equ appendix C2}
			{\rm \uppercase\expandafter{\romannumeral2}}&=\int_0^{+\infty}\left(\int_0^t\mathbb{E} [\Delta G (s, B_{t-s}+u-\alpha(1-2\rho)(t-s)) ] ds\right)du
			\notag\\
			&=\int_0^{+\infty}\left(\int_0^t\left(\int_{-\infty}^{+\infty}\frac{\Delta G (s,v)}{\sqrt{2\pi(t-s)}}e^{-\frac{(v-u+\alpha(1-2\rho)(t-s))^2}{2(t-s)}}dv\right)ds\right)du \\
			&=\int_0^t\left(\int_{-\infty}^{+\infty}\Delta G (s,u)\left(\int_0^{+\infty}\frac{1}{\sqrt{2\pi(t-s)}}e^{-\frac{(v-u-\alpha(1-2\rho)(t-s))^2}{2(t-s)}}dv\right)du\right)ds\notag\\
			&=\int_0^t\left(\int_{-\infty}^{+\infty}\Delta G (s,u)\mathbb{P}\left(B_{t-s}+u+\alpha(1-2\rho)(t-s)\geq 0\right)du\right)ds. \notag
		\end{align}
		This concludes the proof.
	\end{proof}

Now, we are ready to prove Lemma \ref{lemma contraction principle of current}. 

\begin{proof}[Proof of Lemma \ref{lemma contraction principle of current}]
	We only deal with the  case $\beta = 1$ since the remaining cases follow from the same analysis.   It suffices to show that, if  $\mathcal{Q}^\beta_{t_m}(\mu)<+\infty$ and $\int_0^{+\infty} [\mu(t_i,u)-\mu(0, u) ] du=r_i$ for all $1\leq i\leq m$, then 
	\[
	\mathcal{Q}^\beta_{t_m} (\mu)\geq \frac{1}{2} \mathbf{r}^T A^{-1} \mathbf{r}.
	\]
	
	We first prove the above inequality for $\mu$ such that
	\begin{equation}\label{equ condition 5.1}
		\varphi \in C^\infty_c(\R) \text{~and~} G \in C_c^{1,+\infty}([0, T]\times \R),
	\end{equation}
where $\varphi$ and $G$ were identified in Lemma \ref{lemma finite rate}. Let $\tilde{G}_l\in \Scal(\R)$ be the approximation of $G_l$ introduced before \eqref{exp tight error}. By Lemma \ref{lemma 5.1},
	\[
	\mathcal{Q}^\beta_{t_m} (\mu)\geq \frac{1}{2} (\mathbf{r}^l)^T\left(\Sigma_{\{t_k\}_{k=1}^m}(\tilde{G}_l)\right)^{-1}\mathbf{r}^l,
	\]
	where $\mathbf{r}^l=(r_1^l,\ldots, r_m^l)^T \in \R^m$ with $r_i^l=\int_{\R}(\mu(t_i, u)-\varphi (u))\tilde{G}_l(u)du$ for all $1\leq i\leq m$.
	We claim that
	\begin{equation}\label{equ claim proof of Lemma 5.2}
		\lim_{l\rightarrow+\infty} r_i^l= r_i.
	\end{equation}
Then, by \eqref{cov converg},
	\[
	\mathcal{Q}_{t_m} (\mu)\geq \limsup_{l\rightarrow+\infty} \frac{(\mathbf{r}^l)^T\left(\Sigma_{\{t_k\}_{k=1}^m}(\tilde{G}_l)\right)^{-1}\mathbf{r}^l}{2} = \frac{1}{2} \mathbf{r}^T A^{-1} \mathbf{r}.
	\]

	Now we prove \eqref{equ claim proof of Lemma 5.2}.  
	As in the proof of Lemma \ref{lemma 5.3}, for any $H\in L^2(\R)$, 
	\begin{align}\label{equ 5.3}
		\int_{\R}\left(\mu(t,u)-\varphi (u)\right)H(u)du=&\int_{\R}\varphi (u)\left(\mathbb{E}[H(B_t+u+\alpha(1-2\rho)t)]-H(u)\right)du \\
		&-\int_0^t\left(\int_{\R}\partial_{uu}^2G (s,u)\mathbb{E}[H\left(B_{t-s}+\alpha(1-2\rho)(t-s)+u\right)] du\right)ds\notag. 
	\end{align}
	By Equation \eqref{equ 5.3}, Cauchy-Schwarz inequality and the fact that $G \in C_c^{1,+\infty}([0, T]\times \R)$, 
	\[
	\lim_{l\rightarrow+\infty}\int_{\R}(\mu(t, u)- \varphi (u))(\hat{G}_l(u)-G_l(u))du=0. 
	\]
	Hence, to check Equation \eqref{equ claim proof of Lemma 5.2}, we only need to show that
	\begin{equation}\label{equ 5.4}
		\lim_{l\rightarrow+\infty}\int_{\R}(\mu(t, u)-\varphi (u)) G_l (u)du=\int_0^{+\infty}\mu(t, u)-\varphi (u) du.
	\end{equation}
	By Lemma \ref{lemma 5.3} and Equation \eqref{equ 5.3},
	\begin{align}\label{equ 5.5}
		&\int_{\R}(\mu(t, u)-\mu(0,u))G_l(u)du-\int_0^{+\infty}\mu(t, u)-\mu(0,u)du \notag\\
		&=\int_{\R} \varphi (u) F_l(u)du-\int_0^t\left(\int_{\R}\partial_{uu}^2 G (s,u)Q_l(s,u)du\right)ds,
	\end{align}
	where
	\begin{align*}
		F_l(u) &=\mathbb{E}[G_l(B_t+u+\alpha(1-2\rho)t)]-G_l(u)-R_t(u),\\
		Q_l(s,u)&=\mathbb{E}[G_l\left(B_{t-s}+\alpha(1-2\rho)(t-s)+u\right)]-\mathbb{P}(B_{t-s}+u+\alpha(t-s)(1-2\rho)\geq 0).
	\end{align*}
	According to the definition of $G_l$, we have
	\begin{align*}
		|Q_l(s,u)|&\leq \frac{\mathbb{E}|B_{t-s}+u+\alpha(1-2\rho)(t-s)|}{l}+\mathbb{P}\left(B_{t-s}+u+\alpha(t-s)(1-2\rho)\geq l\right).  
	\end{align*}
	Therefore, $\lim_{l\rightarrow+\infty}Q_l (s,u)=0$ uniformly on any compact set in $ [0, t]\times \R$. Similarly, $\lim_{l\rightarrow+\infty}F_l=0$ uniformly on any compact set in $\R$. Then, Equation \eqref{equ 5.4} holds according to Equation \eqref{equ 5.5} and the fact that $\varphi$ and $G$ have compact support. 
	
	For $\mu$ not satisfying \eqref{equ condition 5.1}, the idea is to approximate $\varphi (\mu)$ and $G (\mu)$ by $\varphi^\varepsilon \in C^\infty_c (\R)$ and $G^{\varepsilon} \in C_c^{1,\infty} ([0,T] \times \R)$ respectively.  Precisely speaking, let  $\varphi^\varepsilon  \rightarrow \varphi$ in $L^2(\R)$ and $G^{\varepsilon}  \rightarrow G$ in $\mathbb{H}$ as $\varepsilon \rightarrow 0$.  Let $\mu^\varepsilon \in D ([0, t_m], \Scal^\prime(\R))$ be given in Lemma \ref{lemma finite rate} associated with $\varphi^\varepsilon $ and $G^{\varepsilon}$. 
	Then, by Lemma \ref{lemma finite rate},
	\begin{align}\label{equ limit in proof of Lemma 5.2}
		\lim_{\varepsilon\rightarrow 0}\mathcal{Q}^\beta_{t_m}(\mu^\varepsilon)&=  \lim_{\varepsilon\rightarrow 0} \frac{1}{2\chi (\rho)} \left(\|\varphi^\varepsilon\|_{L^2(\R)}^2+ [G^{\varepsilon},G^{\varepsilon}]\right) \notag\\
		&=\frac{1}{2\chi (\rho)} \left(\|\varphi\|_{L^2(\R)}^2+ [G,G]\right)=\mathcal{Q}_{t_m}(\mu)
	\end{align}
	Since $\mu^\varepsilon$ satisfies \eqref{equ condition 5.1},  we have shown that
	\[
	\mathcal{Q}^\beta_{t_m} (\mu^\varepsilon)\geq \frac{1}{2} (\mathbf{r}^\varepsilon)^T A^{-1} \mathbf{r}^\varepsilon,
	\]
	where $\mathbf{r}^\varepsilon=(r^\varepsilon_1,\ldots, r^\varepsilon_m)^T$ with $r^\varepsilon_i=\int_0^{+\infty}\mu^\varepsilon(t_i,u)-\varphi^\varepsilon (u) du$ for all $1\leq i\leq m$.  Thus,
	\[\mathcal{Q}_{t_m}^\beta (\mu) \geq    \limsup_{\varepsilon\rightarrow 0} \frac{1}{2} (\mathbf{r}^\varepsilon)^T A^{-1} \mathbf{r}^\varepsilon.\]
By Lemma \ref{lemma 5.3},  $\lim_{\varepsilon \rightarrow 0} \mathbf{r}^\varepsilon = \mathbf{r}$, thus concluding the proof. 
\end{proof}

\subsection{MDP for the tagged particle}\label{subsec:tagged} As for the current, we only need to prove exponential tightness and finite dimensional MDP for the tagged particle. Since the orderings of particles are preserved by the dynamics, the tagged particle and the current are related as follows: for $x > 0$, 
\begin{align}\label{tagged current relation}
 \{X^n (t) > x\} = \Big\{J^n_{-1,0} (t) \geq \sum_{y=0}^{x} \eta_y (t) \Big\},  \quad \{X^n (t) < -x\} = \Big\{J^n_{-1,0} (t) <  \sum_{y=-x}^{-1} \eta_{y} (t) \Big\}.
\end{align}
Equivalently, 
\begin{align}
J^n_{-1,0} (t) &= \sum_{x=0}^{X^n(t)-1} \eta_x (t) = \sum_{x=0}^{X^n(t)-1} (\eta_x (t) -\rho) + \rho X^n(t) \quad \text{if}\; J^n_{-1,0} (t) \geq 0;\label{tagged current relation 1}\\
J^n_{-1,0} (t) &= -\sum_{x=X^n(t)}^{-1} \eta_x (t) = - \sum_{x=X^n(t)}^{-1} (\eta_x (t) -\rho) + \rho X^n(t) \quad \text{if}\; J^n_{-1,0} (t) < 0.\label{tagged current relation 2}
\end{align}
Following the proof of \cite[Lemma 5.3]{xue2024sample} and using \eqref{tagged current relation}, one can show that the tagged particle process is exponentially tight. Following the proof of \cite[Lemma 6.2]{xue2024sample} and using \eqref{tagged current relation 1}, \eqref{tagged current relation 2}, one can prove finite dimensional MDP for the tagged particle process. Since the proof is exactly the same, we do not repeat it here. Intuitively, the first terms on the right hand sides of \eqref{tagged current relation 1} and \eqref{tagged current relation 2}, divided by $a_n$, are superexponentially small.  Then, it is straightforward to see that $\mathcal{X}^\beta = \rho^2 \mathcal{J}^\beta$ since the rate functions have quadratic forms.

\appendix

\bibliographystyle{plain}
\bibliography{zhaoreference.bib}

\begin{thebibliography}{10}

\bibitem{balazs2010order}
M.~Bal{\'a}zs and T.~Seppalainen.
\newblock Order of current variance and diffusivity in the asymmetric simple
  exclusion process.
\newblock {\em Ann. of Math.}, 2:1237--1265, 2010.

\bibitem{Decreu1999FBM}
M.~Decreusefond and A.S. \"{U}st\"{u}nel.
\newblock Stochastic analysis of the fractional {Brownian} motion.
\newblock {\em Potential Analysis}, 10:177--214, 1999.

\bibitem{deuschel1989large}
J.~D. Deuschel and D.~W. Stroock.
\newblock {\em Large Deviations}.
\newblock Academic Press, 1989.

\bibitem{Feng2006LDP}
J.~Feng and T.~G. Kurtz.
\newblock {\em Large deviations for stochastic processes}.
\newblock American Mathematical Soc., 2006.

\bibitem{gao2003moderate}
F.~Q. Gao and J~Quastel.
\newblock Moderate deviations from the hydrodynamic limit of the symmetric
  exclusion process.
\newblock {\em Science in China Series A: Mathematics}, 46(5):577--592, 2003.

\bibitem{gonccalves2008central}
P.~Gon{\c{c}}alves.
\newblock Central limit theorem for a tagged particle in asymmetric simple
  exclusion.
\newblock {\em Stochastic Processes and their Applications}, 118(3):474--502,
  2008.

\bibitem{jara2009nonequilibrium}
M.~Jara, C.~Landim, and S.~Sethuraman.
\newblock Nonequilibrium fluctuations for a tagged particle in mean-zero
  one-dimensional zero-range processes.
\newblock {\em Probability theory and related fields}, 145(3):565--590, 2009.

\bibitem{jara2006nonequilibrium}
M.~D. Jara and C.~Landim.
\newblock Nonequilibrium central limit theorem for a tagged particle in
  symmetric simple exclusion.
\newblock In {\em Annales de l'IHP Probabilit{\'e}s et statistiques},
  volume~42, pages 567--577, 2006.

\bibitem{jara2008quenched}
M.~D. Jara and C.~Landim.
\newblock Quenched non-equilibrium central limit theorem for a tagged particle
  in the exclusion process with bond disorder.
\newblock In {\em Annales de l'IHP Probabilit{\'e}s et statistiques},
  volume~44, pages 341--361, 2008.

\bibitem{Johansson00}
K.~Johansson.
\newblock Shape fluctuations and random matrices.
\newblock {\em Communications in mathematical physics}, 209(2):437--476, 2000.

\bibitem{Kipnis86}
C.~Kipnis.
\newblock Central limit theorems for infinite series of queues and applications
  to simple exclusion.
\newblock {\em The Annals of Probability}, 14(2):397--408, 1986.

\bibitem{klscaling}
C.~Kipnis and C.~Landim.
\newblock {\em Scaling limits of interacting particle systems}, volume 320.
\newblock Springer Science \& Business Media, 2013.

\bibitem{kipnis1989hydrodynamics}
C.~Kipnis, S.~Olla, and S.~R.~S. Varadhan.
\newblock Hydrodynamics and large deviation for simple exclusion processes.
\newblock {\em Communications on Pure and Applied Mathematics}, 42(2):115--137,
  1989.

\bibitem{kipnis1986central}
C.~Kipnis and S.~R.~S. Varadhan.
\newblock Central limit theorem for additive functionals of reversible markov
  processes and applications to simple exclusions.
\newblock {\em Communications in Mathematical Physics}, 104(1):1--19, 1986.

\bibitem{komorowski2012fluctuations}
T.~Komorowski, C.~Landim, and S.~Olla.
\newblock {\em Fluctuations in Markov processes: time symmetry and martingale
  approximation}, volume 345.
\newblock Springer Science \& Business Media, 2012.

\bibitem{liggettips}
T.~M. Liggett.
\newblock {\em Interacting particle systems}, volume 276.
\newblock Springer Science \& Business Media, 2012.

\bibitem{liggett99}
T.~M. Liggett.
\newblock {\em Stochastic interacting systems: contact, voter and exclusion
  processes}, volume 324.
\newblock springer science \& Business Media, 2013.

\bibitem{rezakhanlou1994evolution}
F.~Rezakhanlou.
\newblock Evolution of tagged particles in non-reversible particle systems.
\newblock {\em Communications in Mathematical Physics}, 165(1):1--32, 1994.

\bibitem{saada1987tagged}
E.~Said.
\newblock A limit theorem for the position of a tagged particle in a simple
  exclusion process.
\newblock Technical report, 1987.

\bibitem{sethuraman2006diffusive}
S.~Sethuraman.
\newblock Diffusive variance for a tagged particle in $ d \leq 2$ asymmetric
  simple exclusion.
\newblock {\em ALEA Lat. Am. J. Probab. Math. Stat.}, 1:305--332, 2006.

\bibitem{sethuraman2013large}
S.~Sethuraman and S.~R.~S. Varadhan.
\newblock Large deviations for the current and tagged particle in 1d
  nearest-neighbor symmetric simple exclusion.
\newblock {\em The Annals of Probability}, 41(3A):1461--1512, 2013.

\bibitem{sethuraman2023atypical}
S.~Sethuraman and S.R.S. Varadhan.
\newblock Atypical behaviors of a tagged particle in asymmetric simple
  exclusion.
\newblock {\em arXiv preprint arXiv:2311.07800}, 2023.

\bibitem{sethuraman2000diffusive}
S.~Sethuraman, SRS Varadhan, and H.~T. Yau.
\newblock Diffusive limit of a tagged particle in asymmetric simple exclusion
  processes.
\newblock {\em Communications on Pure and Applied Mathematics: A Journal Issued
  by the Courant Institute of Mathematical Sciences}, 53(8):972--1006, 2000.

\bibitem{varadhan1995self}
S.~R.~S. Varadhan.
\newblock Self diffusion of a tagged particle in equilibrium for asymmetric
  mean zero random walk with simple exclusion.
\newblock In {\em Annales de l'IHP Probabilit{\'e}s et statistiques},
  volume~31, pages 273--285, 1995.

\bibitem{xue2024sample}
X.~F. Xue and L.~J. Zhao.
\newblock Sample path {MDP} for the current and the tagged particle in the
  {SSEP}.
\newblock {\em Electronic Journal of Probability}, 29:1--39, 2024.

\bibitem{xue2023moderate}
X.F. Xue and L.J. Zhao.
\newblock Moderate deviations for the current and tagged particle in symmetric
  simple exclusion processes.
\newblock {\em Stochastic Processes and their Applications}, (104218), 2023.

\bibitem{zhao2024moderate}
L.~J. Zhao.
\newblock Moderate deviation principles for the {WASEP}.
\newblock {\em arXiv preprint arXiv:2405.16151}, 2024.

\bibitem{zhao2023motion}
L.~J. Zhao.
\newblock The motion of the tagged particle in the asymmetric exclusion process
  with long jumps.
\newblock {\em Bernoulli}, 30(3):2399--2422, 2024.

\end{thebibliography}
\end{document}